\documentclass[10pt]{article}

\usepackage{lipsum}
\usepackage{amsfonts}
\usepackage{amsthm}

\usepackage{graphicx}
\usepackage{epstopdf}
\usepackage{algorithmic}
\usepackage{mathtools}
\usepackage{xcolor}
  \DeclareGraphicsExtensions{.eps,.pdf,.png,.jpg}

\usepackage{subcaption} 
\usepackage{mathrsfs}
\usepackage[space]{cite}

\newtheorem{theorem}{Theorem}

\newtheorem{proposition}{Proposition}

\begin{document}

\title{A Combined  Parallel-in-time Direct Inverse (ParaDIn)-- Parareal Method for Nonlinear Differential Equations}

\author{Subhash Paudel and Nail K. Yamaleev\footnote{Corresponding author. Department of Mathematics and Statistics, Old Dominion University Norfolk, VA 23529, USA. {\it E-mail:} nyamalee@odu.edu}  \\
{\it\small Old Dominion University}}


\date{}




\maketitle

\begin{abstract}
As has been shown in our previous work,  the parallel-in-time direct inverse (ParaDIn) method introduced by Yamaleev and Paudel in (arXiv: 2406.00878v1,  2024) imposes some constraint on the maximum number of time levels, $N_t$, that can be integrated in parallel.  To circumvent this problem and further increase the speedup,  we combine the  ParaDIn method with the Parareal algorithm to efficiently parallelize the first-order time derivative term in nonlinear partial differential equations discretized by the method of lines.   The main idea of the proposed approach is to use a block-Jacobi preconditioner, so that each block is solved by using the ParaDIn method.  To accelerate the convergence of Jacobi iterations, we use the Parareal method which can be interpreted as a two-level multigrid method in time.  In contrast to the conventional Parareal algorithm whose coarse grid correction step is performed sequentially,  both the coarse- and fine-grid propagators in the proposed approach are implemented in parallel by using the ParaDIn method, thus significantly increasing the parallel performance of the combined algorithm.  Numerical results show that the new combined ParaDIn-Parareal method provides the speedup of up to 124 on 480 computing cores  as compared with the sequential first-order implicit backward difference (BDF1) scheme  for the 2-D nonlinear heat and Burgers equations with both smooth and discontinuous solutions.
%
%
\end{abstract}


\section{Introduction}
\quad \  Parallel-in-time numerical schemes have recently attracted a lot of attention due to their potential of
drastically reducing the computational time and achieving much higher scalability than spatial domain decomposition methods for solving large-scale unsteady problems on modern supercomputers with hundreds of thousands of computing cores.  Despite that various parallel-in-time methods are available in the literature, they have not been used in production codes.
As follows from two comprehensive literature reviews \cite{Gan, OS}, the existing parallel-in-time methods have
not yet reached such level of maturity that is required for practical applications governed by highly nonlinear partial differential equations.

One of the most widely-used parallel-in-time methods that can be directly applied to both linear and nonlinear problems is the Parareal algorithm introduced by Lions,  Maday,  and Turinici in \cite{LMT}.  This method can be considered as an iterative shooting method or a two-level multigrid method with special restriction and prolongation operators in the time domain \cite{LMT, GV}. The main idea of the Parareal method is to divide the entire time interval into subintervals,  so that the discretized governing equations can be independently integrated in each subinterval. To couple local equations defined on the original fine grid in each subinterval,  a computationally inexpensive {\it sequential}  solver is used to find the solution on the coarse temporal grid.
Since the fine-grid solves are independent from each other, they can be performed in parallel.  The difference between the coarse- and fine-grid solutions is propagated sequentially on the coarse grid,  which is called a coarse grid correction step.   The coarse- and fine-grid steps are repeated iteratively until some residual norm becomes smaller than a user-defined tolerance \cite{LMT}. 
For a linear system of ordinary differential equations (ODEs) with a constant coefficient matrix $A \in {R}^{m \times m}$,  it has been proven that the Parareal method converges  if the governing equation is discretized by using an implicit $L$-stable scheme (e.g., the backward Euler scheme) and $A$ is either positive symmetric positive definite (SPD) \cite{GV} or has complex eigenvalues \cite{Wu2}.
Various parallel-in-time algorithms have been developed to enhance the efficiency and stability properties of the classical Parareal method including Parareal methods with spatial coarsening,  Krylov-subspace enhanced Parareal methods, hybrid spectral deferred correction methods, multigrid reduction in time (MGRIT),  parallel exponential integrators (ParaExp) and others \cite{FFKMS,  Min, DM,  Chen, GV, Rup, GGP}.

Despite their success,  there are various factors that affect the performance of the Parareal-type methods including nonlinearity and a type of PDE solved (e.g., its parallel efficiency  significantly deteriorates for hyperbolic PDEs \cite{SR}),  a choice of the coarse- and fine-grid solvers (e.g., the convergence rate significantly decreases for discrete operators with imaginary eigenvalues \cite{GV}),  stiffness caused by input parameters (e.g.,  small values of the viscosity coefficient \cite{GLWYZ}), and others.
One of the main bottlenecks of the Parareal-type methods  is the coarse grid correction step that is carried out sequentially, thus dramatically reducing the speedup of the overall algorithm.  
Recently,  several attempts have been made to overcome this problem.  One of the approaches is to combine the Parareal algorithm with a diagonalization technique introduced in \cite{MR},  which is used to parallelize the coarse grid correction step \cite{Wu}.  Another approach uses a spatially coarsened grid at the coarse grid correction step with the goal of reducing the computational cost as compared with that of the fine-grid propagator \cite{AGR}.  Unfortunately,  the parallel efficiency of these time-parallelization methods deteriorates dramatically if the governing equations are essentially nonlinear and convection-dominated.


We have recently developed a novel approach to parallelization of the implicit first-order backward difference (BDF1) scheme for unsteady nonlinear partial differential equations of arbitrary type in \cite{YP}.  
The global system of nonlinear discrete equations in the space-time domain is solved by the Newton method for all time levels simultaneously. This all-at-once system at each Newton iteration is block bidiagonal, which can be solved analytically by using the Gaussian elimination in a blockwise manner, thus leading to a set of fully decoupled equations associated with each time level.  Note that the product matrices on the left- and right-hand sides of the all-at-once system of equations for all time levels can be efficiently computed in parallel, such that the number of operations performed by each computing core is linear in the number of spatial degrees of freedom. 
Furthermore,  the computational cost of solving each block matrix  is nearly identical to that of the sequential BDF1 scheme at each Newton iteration on each time step if the same direct solver is used for both systems of linear equations.   
This allows for an efficient parallel-in-time implementation of the implicit BDF1 discretization for nonlinear differential equations.  In contrast to the existing parallel-in-time algorithms, the proposed ParaDIn method preserves the quadratic rate of convergence of the Newton method of the corresponding sequential scheme.  As has been shown in \cite{YP},  some upper bound should be imposed
on the total number of time steps $N_t$ that can be integrated in parallel by using the ParaDIn method.
This constraint is due to the fact that the condition number of a product matrix on the left-hand side of the ParaDIn method increases together with the number of time steps. To circumvent this problem, we propose a novel strategy based on combining the ParaDIn
method with the Parareal algorithm, so that both the coarse- and fine-grid correction steps are performed in parallel.  Combining the ParaDIn and Parareal methods allow us to further increase the parallel efficiency of  the ParaDIn method by a factor of 2-4 for 2-D nonlinear heat and Burgers equations with both smooth and discontinuous solutions.

The paper is organized as follows.  In section 1, we present governing equations and the baseline BDF1 scheme.  The ParaDIn method is briefly presented in section 3.  Then,  we discuss the block-Jacobi method and its limitations in section 4.  We present the new combined ParaDIn-Parareal method and study its properties in sections 5-6.  Numerical results demonstrating the parallel efficiency and scalability of the proposed method are presented in section 7 and conclusions are drawn in section 8.

\section{Setting of the problem }
\label{GE}

In the present analysis,  the following 2-D scalar nonlinear conservation law equation is considered as a model problem:
\begin{equation}
\label{eq:CLE}
\frac{\partial u}{\partial t} + \frac{\partial f}{\partial x} + \frac{\partial g}{\partial y} = 
\frac{\partial}{\partial x}\left( \mu\frac{\partial u}{\partial x}\right) + \frac{\partial}{\partial y}\left( \mu\frac{\partial u}{\partial y}\right),  
\ \forall (x, y) \in \Omega, \ t \in (0, T_f],
\end{equation}
where $f(u)$ are $g(u)$ are inviscid fluxes associated with the $x$ and $y$ directions  respectively,  $\mu(u) \ge 0$ is a nonlinear viscosity coefficient,  and $\Omega = \left\{(x,y) | \right.$ $\left. \ x_L < x < x_R, y_L < y <y_R \right\}$.  
The above equation is subject to the following  initial condition: 
$$ u(x, y, 0) = u_0(x, y, 0),$$
and Dirichlet boundary conditions:
\begin{equation}
\label{eq:BC}
\begin{array}{lc}
u(x_L, y, t) = b(x_L, y, t), \ \ \ & u(x_R, y, t) = b(x_R,  y, t), \\
u(x, y_L, t) = b(x, y_L, t), \ \ \ & u(x, y_R, t) = b(x,  y_R, t),
\end{array}
\end{equation}
where $b$ and $u_0$ are bounded functions  in $L_2 \cap L_{\infty}$ for which Eqs. (\ref{eq:CLE}-\ref{eq:BC}) are well posed.
%

We discretize Eq. (\ref{eq:CLE}) on a uniform Cartesian grid as follows:
\begin{equation}
\label{eq:FDS}
\begin{split}
 \frac{u_{j,i}^n-u_{j,i}^{n-1}}{\tau_n} + &\frac{f_{j,i+1}^n - f_{j,i-1}^n }{2 h_x} + \frac{f_{j+1,i}^n - f_{j-1,i}^n }{2 h_y} \\
                                                              = & \frac{\mu_{j, i+\frac 12} \left(u_{j,i+1}^n - u_{j,i}^n\right) -  \mu_{j, i-\frac 12} \left(u_{j,i}^n - u_{j,i-1}^n\right)}{h^2_x} \\
                                                              + & \frac{\mu_{j+\frac 12,i} \left(u_{j+1,i}^n - u_{j,i}^n\right) -  \mu_{j-\frac 12,i} \left(u_{j,i}^n - u_{j-1,i}^n\right)}{h^2_y},
\end{split}
\end{equation}
where $\tau_n = t^n - t^{n-1}$ is a time step size,  $h_x$ and $h_y$ are grid spacings in the $x$ and $y$ directions, respectively.  
Using the Newton method at each time step, the nonlinear discrete equations (\ref{eq:FDS}) can be solved starting from the time level $n=1$ and marching forward in time until $n = N_t$.  Since the solution $u^{n-1}$ at the previous time level is required to advance the solution to the next time level,  this time integration method is inherently sequential and cannot be parallelized in time.

\section{Parallel-in-time direct inverse method}
\label{GIT}

To parallelize the implicit BDF1 scheme in time,  Eq. (\ref{eq:FDS}) is considered as a single global space-time system of equations:
\begin{equation}
\label{eq:GIT}
\left\{
\begin{array}{l}
\frac{{\bf u}^1 -{\bf u}^{0}}{\tau_n} + {\bf F}\left({\bf u}^1\right) = {\bf q}^1  \\
\quad\quad\quad \dots \\
\frac{{\bf u}^n -{\bf u}^{n-1}}{\tau_n} + {\bf F}\left({\bf u}^n\right) = {\bf q}^n \\
\quad\quad\quad  \dots \\
\frac{{\bf u}^{N_t} -{\bf u}^{N_t-1}}{\tau_{N_t}} + {\bf F}\left({\bf u}^{N_t}\right) = {\bf q}^{N_t},
\end{array}
\right.
\end{equation}
where ${\bf u}^n = \left[u_{1,1}^n,  \dots , u_{1,N_x}^n, \dots, u_{N_y,  N_x}^n \right]^T$ is a solution vector,  ${\bf F}$ is a nonlinear discrete spatial operator associated with the inviscid and viscous terms in Eq. (\ref{eq:FDS}),   and ${\bf q}$ includes the contribution from the initial and boundary conditions.  Using the Newton method to solve the all-at-once system given by Eq. (\ref{eq:GIT}),  we have
\begin{equation}
\label{eq:GlobalNewton}
\left[
\begin{array}{cccc}
A_1      &   0        & \dots   &       0      \\
-I         & A_2      & \ddots &  \vdots   \\
\vdots  & \ddots  & \ddots &      0       \\
0         & \dots    &     -I     &  A_{N_t} \\
\end{array}
\right]
\left[
\begin{array}{c}
\Delta{\bf u}^1         \\
\Delta{\bf u}^2         \\
\vdots                        \\
\Delta{\bf u}^{N_t}   \\
\end{array}
\right] = 
\left[
\begin{array}{c}
{\bf r}^1         \\
{\bf r}^2         \\
\vdots             \\
{\bf r}^{N_t}   \\
\end{array}
\right],
\end{equation}
where ${\bf u}_{k+1}^n = {\bf u}_k^n + \Delta{\bf u}^n$,  ${\bf u}_k^n \in  \mathbb{R}^{N_s}$ ($N_s = N_x N_y$) is the solution vector on $n$-th time level at $k$-th Newton iteration,  $A_n = I + \tau_n \left(\frac{\partial{\bf F}}{\partial {\bf u}}\right)^n_k$, $n=1,\dots, N_t$ are $N_s\times N_s$ Jacobian matrices,  ${\bf r}^n$ is a residual vector associated with the $n$-th time level, 
and $I$ is the $N_s\times N_s$ identity matrix.  Note that the equations (\ref{eq:GlobalNewton}) are fully coupled due to the subdiagonal of the left-hand-side matrix and cannot be directly solved in parallel.

In our previous work \cite{YP}, we introduced a new parallel-in-time direct inverse (ParaDIn) method for parallelization of the implicit BDF1 scheme in time for unsteady nonlinear partial differential equations of arbitrary type.   The key idea of the ParaDIn method is based on the observation that Eq.  (\ref{eq:GlobalNewton}) can be recast in a fully decoupled form and integrated in parallel as presented in the following theorem.
\begin{theorem}
\label{th:Decouple}
If matrices $A_i$, $i=1,\dots,N_t$  are nonsingular,  i.e., $\det A_i \ne 0, \ \forall i$, then Eq.~(\ref{eq:GlobalNewton}) can be decoupled as follows:
\begin{equation}
\label{eq:Dec}
\left[
\begin{array}{cccc}
A_1      &   0        & \dots   &       0      \\
0          & \prod\limits_{i=1}^{2}A_i      & \ddots &  \vdots   \\
\vdots  & \ddots  & \ddots &      0       \\
0         & \dots    &     0     &  \prod\limits_{i=1}^{N_t}A_i \\
\end{array}
\right]
\left[
\begin{array}{c}
\Delta{\bf u}^1         \\
\Delta{\bf u}^2         \\
\vdots                        \\
\Delta{\bf u}^{N_t}   \\
\end{array}
\right] = 
\left[
\begin{array}{c}
{\bf r}^1         \\
A_1 {\bf r}^2  + {\bf r}_1      \\
\vdots             \\
\sum\limits_{j=2}^{N_t}  \prod\limits_{i=1}^{j-1} A_i {\bf r}_j +{\bf r}^{1}   \\
\end{array}
\right],
\end{equation}
so that the solution of Eq.~(\ref{eq:Dec}) is unique and identical to that of Eq. (\ref{eq:GlobalNewton}).
\end{theorem}
\begin{proof}
The existence and uniqueness of the solution of Eq. (\ref{eq:Dec}) follow immediately from the fact that $\det\left( \prod\limits_{i=1}^{j} A_i \right) \ne 0$
for $j=1, \dots, N_t$  which is a direct consequence of $\det A_i \ne 0$ for $i=1, \dots,  N_t$.  

Now,  let us prove that the solutions of Eqs. (\ref{eq:GlobalNewton}) and (\ref{eq:Dec}) are identical to each other.  
We begin by noting that the first equations in both systems are identical  and fully decoupled form the remaining equations, which implies that their solutions are equal to each other. 
Multiplying both sides of the 2nd equation in Eq.~(\ref{eq:GlobalNewton}) by the nonsingular matrix $A_1$ and adding the first equation  yields
$$
\prod\limits_{i=1}^2 A_i \Delta{\bf u}_2 = A_1 {\bf r}_2 +  {\bf r}_1,
$$
which is identical to the second equation in Eq.~(\ref{eq:Dec}).   Adding the above equation to the 3rd equation multiplied by  by $\prod\limits_{i=1}^2A_i$ in Eq.~(\ref{eq:GlobalNewton}),  we have
$$
\prod\limits_{i=1}^3 A_i \Delta{\bf u}_3 = \prod\limits_{i=1}^2 A_i {\bf r}_3 + A_1{\bf r}_2 + {\bf r}_1,
$$
which is identical to the 3rd equation in Eq.~(\ref{eq:Dec}).
Repeating this procedure recursively $N_t-3$ times, we recover all the equations in Eq. (\ref{eq:Dec}).  Since at each step of this procedure we perform the operations that do not change the solution of the original system of equations,  the solution of Eq. (\ref{eq:Dec}) is identical to the solution of Eq.~(\ref{eq:GlobalNewton}).
\end{proof}

This ParaDIn methodology has several unique properties.  If the Newton method given by Eq. (\ref{eq:GlobalNewton}) converges quadratically,  then the new parallel-in-time method based on Eq. (\ref{eq:Dec}) preserves this optimal convergence rate,  because Eqs.  (\ref{eq:GlobalNewton}) and (\ref{eq:Dec}) are equivalent to each other and have identical solutions as follows from Theorem \ref{th:Decouple}.
Also, the computational cost of solving each equation in Eq. (\ref{eq:Dec}) is nearly identical to that of the sequential BDF1 scheme at each Newton iteration on the corresponding time step,  if the same direct solver is used for solving both systems of linear equations. Another advantage of the proposed parallel-in-time integration scheme is that it preserves the original discretization and can be directly combined with standard spatial domain-decomposition algorithms, thus promising a much higher speedup on large computer platforms as compared with the current state-of-the-art methods based on parallelization of the spatial discretization alone. Furthermore, the new BDF1 scheme given by Eq. (\ref{eq:Dec}) provides the highest level of parallelism in time,  because each time level is computed on its own computing core in parallel. 
Since the original and fully decoupled systems of equations (\ref{eq:GlobalNewton}) and (\ref{eq:Dec}) are equivalent to each other regardless of what differential operator is approximated by the Jacobian matrices $A_i,  i=1,\dots,N_t$,  the ParaDIn method can be used for solving both linear and nonlinear differential equations of arbitrary type and for any spatial discretizations including finite element,  finite volume,  spectral collocation, and finite difference methods.

To make the computational cost of calculating the left- and right-hand sides of  Eq. (\ref{eq:Dec}) much lower than that of solving the system of linear equations on each time level/computing core, the following parallel algorithm has been developed in \cite{YP}.  The key idea of this algorithm is that instead of calculating each matrix $\prod\limits_{i=1}^n A_i $ individually on the corresponding computing core $C_n$,  we calculate each product matrix  and the corresponding right-hand side by using all $C_1,\dots,C_{N_t}$ processors in parallel.  This parallel algorithm is briefly outlined next.

First,  the solution vectors ${\bf u}_1$ and ${\bf u}_2$ are sent to all computing cores. Then, we partition the matrix $A_1$ into $N_t$ rectangular matrices,  such that each submatrix containing $m=N_s/N_t$ rows of $A_1$ is formed on the corresponding computing core.  In other words,  the rows from $1$ to $m$ are computed on a core $C_1$,  rows from $m+1$ to $2m$ are computed on a core $C_2$ and so on up to the last computing core $C_{N_t}$.  
Along with the rows of $A_1$, we also compute all nonzero entries of the matrix $A_2$ on each computing core, which requires only $O(N_s)$ operations. 
As a result,  each core has the corresponding batch of rows of $A_1$ and all columns of $A_2$.  Taking into account the sparsity of the Jacobian matrices $A_1$ and $A_2$,  only nonzero entries of the product matrix $\prod\limits_{i=1}^2 A_i$ are computed to reduce the computational cost.  After this multiplication, each core has the corresponding $m$ rows of the product matrix,  which are then sent to the core $C_2$ to form the entire matrix $\prod\limits_{i=1}^2 A_i$.  Similarly,  we can calculate the product of $m$ rows of $A_1$ and  the vector $\tilde{\bf r}_2$ on each computing core.  Repeating this multiplication procedure recursively $l$ times,  the corresponding rows of the product matrix $P^l = \prod\limits_{i=1}^l A_i $ and the right-hand-side 
$\tilde{\bf r}^l$ of Eq. (\ref{eq:Dec})  are calculated on each computing core as follows:
\begin{equation}
\label{eq:Prod}
\begin{split}
p^{l}_{ij} &= \sum\limits_{k\in K_{ls}} p^{l-1}_{ik} a^{l}_{kj},                            \quad l=2,\dots, N_t, \ i = (s-1) m + 1, \dots, s m \\
\tilde{r}^{l}_{i}  &= \sum\limits_{k\in K_{l}} p^{l-1}_{ik} r^{l}_{k} + \tilde{r}^{l-1}_{i},  \\
\end{split}
\end{equation}
where $a^l_{kj}$,  $1\le k, j \le N_s$ are entries of the matrix $A^l$,  $\tilde{r}^{l}_{i}$ and $r^{l}_{k}$ are entries of the right-hand sides in the $l$-th equation in Eqs. (\ref{eq:Dec}) and (\ref{eq:GlobalNewton}), respectively,  and $K_{ls}$ is a set of indices of nonzero entries of $P^{l-1}$ and $A_l$.  

As has been proven in Theorem 2 in \cite{YP},  if $N_t < N_s^{1/2}$, the total number of operations performed by each processor for computing all product matrices $P^l=\prod\limits_{i=1}^l A_i, l=1, \dots, N_t$ and right-hand sides in Eq.~(\ref{eq:Dec}) by the above parallel algorithm is $O(N_t^2 N_s)$, where $N_t$ and $N_s$  are the total numbers of time steps and spatial degrees of freedom, respectively.  
Therefore,  assuming that the computational cost $W_{\rm sol} $ of solving the system of $N_s$ linear equations at each time level is dominant,  i.e.,  $W_{\rm sol} \gg O(N_t^2 N_s) \gg N_s$ and significantly higher than the  communication cost,  the speedup provided by ParaDIn method can evaluated as follows (see section 6 in  \cite{YP}):
\begin{equation}
\label{eq:speedup}
S = \frac{N_t}{\frac{N_t}{c_f^3}+1},
\end{equation}
where $c_f$ is a coarsening factor by which the temporal and spatial grids are coarsened to compute an initial guess for the global-in-time Newton method given by Eq. (\ref{eq:GlobalNewton}).
If $N_t \ll c_f^3$,  a nearly ideal speedup of $N_t$ can be achieved with the ParaDIn method for 2-D time-dependent problems.

\section{Block-Jacobi Method}
\label{section:blockJacobi}
As has been shown in \cite{YP},  the condition number of the product matrix $\prod\limits_{i=1}^n A_i$ in the left-hand side of Eq. (\ref{eq:Dec}) increases together with $n$,  thus  imposing a constraint on the maximum number of time steps $N_t$ that can be integrated in parallel and limiting the scalability of the ParaDIn  algorithm.   
 A straightforward approach to circumvent the condition number constraint is to use a block-Jacobi method.  Indeed, instead of using the ParaDIn algorithm for all $N_t$ time levels,  we partition the entire time domain into $M$ subintervals (blocks), such that the ParaDIn method is used for each subinterval individually.   Using the ParaDIn method within each block,  the all-at-once system of equations  (\ref{eq:GlobalNewton}) can be written in the following form:
\begin{equation}
\label{eq:Jacobi}
\left[
\begin{array}{cccc}
B_1      &   0        & \dots    &       0      \\
0          & B_2      & \ddots &  \vdots   \\
\vdots  & \ddots  & \ddots &      0       \\
0         & \dots    &     0       &  B_{M} \\
\end{array}
\right]
\left[
\begin{array}{c}
{\bf v}^1         \\
{\bf v}^2         \\
\vdots                        \\
{\bf v}^{M}   \\
\end{array}
\right] = 
\left[
\begin{array}{c}
{\bf r}_1         \\
{\bf r}_2         \\
\vdots             \\
{\bf r}_{M}   \\
\end{array}
\right],
\end{equation}
where the system of equations associated with each block $B_m$, $m=2,\dots, M$  and vector ${\bf v}^m$  at the $k$-th Jacobi iteration is given by
\begin{equation}
\label{eq:Jacobi_block}
\left\{
\begin{array}{rl}
A_{(m-1)J+1}  {\bf v}^{(m-1)J+1}_k        = & {\bf r}_{(m-1)J+1} +  {\bf v}^{(m-1)J}_{k-1} \\
\vdots & \\
 \prod\limits_{i=(m-1)J+1}^{m J} A_i {\bf v}^{m J}_k = &  \sum\limits_{j=(m-1)J+2}^{m J}  \prod\limits_{i=(m-1)J+1}^{j-1} A_i {\bf r}_j  + {\bf r}_{(m-1)J+1},
\end{array}
\right. 
\end{equation}
where  ${\bf v}^m = \left[ \Delta{\bf u}^{(m-1) J+1},  \dots,  \Delta{\bf u}^{m J} \right]^T$,
$J=\frac{N_t}M$,  and $A_i = I + \tau_i \left(\frac{\partial{\bf F}}{\partial {\bf u}}\right)^i$.  
As evident from Eqs.(\ref{eq:Jacobi}-\ref{eq:Jacobi_block}), the neighboring blocks $B_{m-1}$ and $B_{m}$ are coupled with each other only via the single term ${\bf v}^{(m-1) J}$.
Therefore,  by taking this coupling term ${\bf v}^{(m-1) J}$ from the previous $(k-1)$-th Jacobi iterations,  the system of equations (\ref{eq:Jacobi}-\ref{eq:Jacobi_block}) becomes fully decoupled and can be solved in parallel.

It should be emphasized that the maximum number of factors in the product matrix on the left-hand side of each block is $J=\frac{N_t}{M}$, which is $M$ times smaller than that of the original ParaDIn method used for all time levels $N_t$,  thus drastically reducing the condition number of the product matrices associated with each time subinterval.  One of the key advantages of this approach is that all blocks can be solved simultaneously and all equations within each block are decoupled and can also be solved in parallel.  Thus, the above block-Jacobi method allows us to increase the total number of computing cores and, consequently,  the total number of time steps that can be computed in parallel, by a factor of $M$ (where $M$ is the number of subdomains or batches in time) as compared with the baseline ParaDIn method used over the entire time interval. 

Despite its simplicity and ability to reduce the condition number,  this block Jacobi method suffers from slow convergence, which drastically reduces the parallel efficiency of the overall algorithm.  Indeed,  the block Jacobi method can be written in the following form:
\begin{equation}
    {\bf v}^n_k = T  {\bf v}^n_{k-1} + {\bf b}
\end{equation}
where $T$ is the iteration matrix and ${\bf b}$ is a vector of known values.  The iteration matrix for Eqs. (\ref{eq:Jacobi}-\ref{eq:Jacobi_block}) is a block diagonal matrix with the blocks of the following form:
\begin{equation}
\label{eq: time parallel block jacobi}
T_m = \left[
\begin{array}{cccc}
0             &              &  \dots         &       0      \\
A_{(m-1)J)+1}^{-1} &  0         &                    &  \vdots   \\
 \vdots    & \ddots  & \ddots        &               \\
0             &  \dots   &     A_{m J}^{-1}&  0    \\
\end{array}
\right].  \quad {\rm for} \ m=2,\dots,M.
\end{equation}
To simplify the analysis, we consider the 1-D  linear heat equation $v_t = \mu v_{xx}$ with the constant viscosity coefficient $\mu = \mu_0$, which is discretized by using the BDF1 approximation in time and the 3-point discrete Laplacian in space on a uniform grid both in space and time. Therefore,  $A_1 = A_2 = \dots= A_{N_t} = A$, where $A$ is a tridiagonal matrix.
Under these assumptions and ignoring the communication time,  it has been proven in \cite{DMD}, that the total computational cost $W_{\rm par}$ to solve Eq. (\ref{eq:Jacobi}) by the conventional block Jacobi method on $M$ processors in parallel  is
\begin{equation}
\label{eq:par_time_Jacobi}
W_{\rm par}  \geq \frac{M W_1}{\pi^2}\ln\left(\frac{1}{\epsilon}\right) =  \frac{W_{\rm ser}}{\pi^2}\ln\left(\frac{1}{\epsilon}\right),
\end{equation}
where $W_1$ is the computational cost of solving this discretized heat equation by the serial sequential BDF1 scheme over one time subinterval (i.e., one block),  $W_{\rm ser}$ is the total computational cost of the serial BDF1 scheme,  and $\epsilon$ is a tolerance.  Based on the estimate given by Eq. (\ref{eq:par_time_Jacobi}), we can conclude that even if the communication cost is negligibly small, the total runtime to solve Eq. (\ref{eq:Jacobi}) in parallel on $M$ processors is asymptotically comparable to the serial runtime.  If the tolerance is set to be sufficiently small ($\epsilon < O(10^{-5}) $),  then the computational costs of the parallel and serial BDF1 schemes become nearly the same, thus providing practically no speedup regardless of the number of processors used.

It should be noted that the same slow convergence is observed for the block Jacobi method given by Eq.~(\ref{eq:Jacobi}).  Furthermore,  if the tolerance is set to be sufficiently small, then this block Jacobi method converges precisely in $M$ iterations.
\begin{proposition}
The block Jacobi method given by Eqs. (\ref{eq:Jacobi}-\ref{eq:Jacobi_block}) converges to the exact discrete solution of  Eq.~(\ref{eq:GlobalNewton}) in precisely $M$ iterations.
\end{proposition}
\begin{proof}
As follows from Theorem \ref{th:Decouple},  the ParaDIn method provides the exact solution of  the linear system of equations associated with each block $B_m, m=1,\dots, M$. Therefore, after the first Jacobi iteration,  $[{\bf v}_1^1,{\bf v}_1^2,\dots,{\bf v}_1^J]^T$ is equal to the exact discrete solution of  Eqs.(\ref{eq:Jacobi}-\ref{eq:Jacobi_block}) on the first time subinterval.  Since ${\bf v}_1^J$ is exact at the second Jacobi iteration, the entire right-hand side vector of the second block is exact as well.  Thus,  $[{\bf v}_2^{J+1},{\bf v}_2^{J+2},\dots,{\bf v}_2^{2J}]^T$ is discretely exact after the second Jacobi iteration.  Repeating this procedure recursively $M$ times,  it follows that the entire solution vector $[{\bf v}_k^1,{\bf v}_k^2,\dots,{\bf v}_k^{N_t}]^T$ becomes equal to the exact solution of Eq.~(\ref{eq:GlobalNewton}) at the $M$-th Jacobi iteration.
\end{proof}

Since the block Jacobi method given by Eqs. (\ref{eq:Jacobi}-\ref{eq:Jacobi_block}) converges in $M$ iterations, the overall computational cost of this parallel-in-time method is comparable to that of using the baseline ParaDIn method in a blockwise fashion sequentially and may become even higher due to the communication overhead between the computing cores.  So, this straightforward block Jacobi method is not suitable for temporal parallelization, because it practically provides no speedup as compared with the original ParaDIn method. 
We propose to overcome this problem and accelerate the convergence of block-Jacobi iterations by combining the ParaDIn method with the Parareal algorithm which can be interpreted as a two-level multigrid method in time \cite{GV}.   This new combined ParaDIn--Parareal method is presented next.

\section{Combined ParaDIn--Parareal method}
\label{sec:PP}
In the previous section, we have shown that some mechanism is needed to propagate information over the entire time domain at each Jacobi iteration.  To address this question,  we propose to combine the ParaDIn algorithm (see Section \ref{GIT}) with a multigrid method in time to accelerate convergence for solving the system of linear equations  (\ref{eq:GlobalNewton}) at each Newton iteration.
There are several variants of the multigrid method in time including Parareal \cite{LMT},  MGRIT \cite{FFKMS},  PFASST \cite{Min} and others \cite{GN, OS}. In the present analysis, we use the Parareal method, which can be interpreted as a two-level multigrid method in time with special restriction and prolongation operators \cite{GV}.  

There are two main variants of the Parareal method
for nonlinear problems. The first approach is to use the sequential scheme (e.g., the BDF1 scheme given by Eq. (\ref{eq:FDS})) with coarse and fine time steps as a coarse- and fine-grid propagators, respectively.  It should be noted that in this case, the  nonlinear discrete equations are solved by the Newton method  at each time step sequentially.  Despite the simplicity of this approach, there are no theoretical results showing that the classical Parareal method converges for essentially nonlinear problems.  Our numerical results show that this variant of the Parareal method diverges for the 2-D nonlinear heat and Burgers equations considered in the present analysis.  An alternative approach is to consider the discrete equations (\ref{eq:FDS}) as a global system of equations in the space-time domain and solve it by using the Newton method as described in Section \ref{GIT}.  At each Newton iteration,  the global-in-time all-at-once system of linear equations can be solved by the Parareal method.  Unlike its nonlinear counterpart,  the theory of the Parareal method for linear equations is well established (e.g., see \cite{GV, Wu2}).  One of the drawbacks of the conventional Parareal method is the fact that the coarse-grid correction and the time marching in the fine-grid propagator are performed sequentially,  thus significantly limiting its parallel efficiency. We propose to circumvent this problem by implementing the coarse- and fine-grid propagators  in a parallel-in-time manner by using the new ParaDIn method outlined  in Section \ref{GIT}.  

Considering the BDF1 scheme given by Eq.~(\ref{eq:FDS}) as  a global system of equations in the space-time domain and using the Newton method,  the problem reduces to the all-at-once system of linear equations (\ref{eq:GlobalNewton}) for the global vector ${\bf v}^n = \Delta{\bf u}^n = \left[\Delta u_{1,1}^n,  \dots , \Delta u_{1,N_x}^n, \dots, \Delta u_{N_y,  N_x}^n \right]^T$.  
Similar to the block Jacobi method presented in Section \ref{section:blockJacobi}, we begin by partitioning the entire time domain into $M$ blocks and recasting the all-at-once linear system of equations  in the block-matrix form given by Eqs. ~(\ref{eq:Jacobi}--\ref{eq:Jacobi_block}). 
At each Newton iteration,  to solve this system of linear equations (\ref{eq:Jacobi}--\ref{eq:Jacobi_block}) and improve the parallel efficiently,  we combine the ParaDIn method with the Parareal algorithm  which is an iterative method based on coarse- and fine-grid schemes called propagators and consists of the following three steps: 1) initialization,  2) fine-grid solve, and 3) coarse-grid correction. In contrast to the conventional Parareal method that uses some sequential schemes (e.g.,  the BDF1 scheme) on the coarse and fine temporal grids as the coarse- and fine-grid propagators, respectively, we use the parallel-in-time ParaDIn--BDF1 scheme on both the coarse and fine grids as described in sections \ref{init}--\ref{CGC}.

\subsection{Step 1: Initialization} 
\label{init}
First, we initialize the solution $\Delta{\bf u}_c^1,\dots,\Delta{\bf u}_c^M$ on the coarse temporal grid $\{t_0, t_1,\dots, t_{M}\}$, where $t_i = i \Delta t$ and $\Delta t=\frac{T_f}M$, by solving the following system of linear equations:
\begin{equation}
\label{eq:Init}
\left[
\begin{array}{cccc}
A^c_1      &   0        & \dots   &       0      \\
-I         & A^c_2      & \ddots &  \vdots   \\
\vdots  & \ddots  & \ddots &      0       \\
0         & \dots    &     -I     &  A^c_{M} \\
\end{array}
\right]
\left[
\begin{array}{c}
(\Delta{\bf u}_c^1)_0         \\
(\Delta{\bf u}_c^2)_0         \\
\vdots                        \\
(\Delta{\bf u}_c^{M})_0   \\
\end{array}
\right] = 
\left[
\begin{array}{c}
{\bf r}^c_1         \\
{\bf r}^c_2         \\
\vdots             \\
{\bf r}^c_{M}   \\
\end{array}
\right],
\end{equation}
where $A^c_i = I + \Delta t \left(\frac{\partial{\bf F}}{\partial {\bf u}}\right)^i$ and ${\bf r}^c_i$ are the Jacobian matrix and the residual vector associated with the $i$-th time level on the coarse temporal grid.   Without loss of generality,  it is assumed that the temporal coarse and fine grids are both uniform, i.e.,  $\Delta t_1 =\dots=\Delta  t_{M}=\Delta t=\frac{T_f}M$ and $\tau_1 =\dots=\tau_{N_t} = \tau=\frac{T_f}{N_t}$. Using the ParaDIn method for solving Eq. (\ref{eq:Init}) leads to
\begin{equation}
\label{eq:ParaDin_init}
\left\{
\begin{array}{rl}
A^c_1  (\Delta{\bf u}^1_c)_0        = & {\bf r}^c_1 \\
\prod\limits_{i=1}^{2} (\Delta{\bf u}^2_c)_0 = &A^c_1 {\bf r}^c_2 + {\bf r}^c_1 \\
\vdots & \\
 \prod\limits_{i=1}^{M} \Delta A^c_i (\Delta{\bf u}^{M}_c)_0 = &  \sum\limits_{j=2}^{M}  \prod\limits_{i=1}^{j-1} A^c_i {\bf r}^c_j  + {\bf r}^c_1, \\
\end{array}
\right.
\end{equation}
where the subscript $0$ is the Parareal iteration index.
Unlike the classical Parareal method that initializes the solution on the coarse temporal mesh sequentially,  in the proposed method, the initialization is performed in parallel because the system of equations (\ref{eq:ParaDin_init}) is fully decoupled.  Note that the coarse grid is selected so that the number of grid points $M$ satisfies the constraints of the ParaDIn method discussed in Section \ref{GIT}.

\subsection{Step 2: Fine-grid Propagator}
 After the initialization step,   $(\Delta{\bf u}^1_c)_0,\dots,(\Delta{\bf u}^{M-1}_c)_0$  are used as initial conditions for the fine-grid propagator in the $B_2,\dots,B_M$ blocks in Eq.~(\ref{eq:Jacobi}), respectively.  Thus, the fine-grid propagator step for all $M$ blocks can be written as follows:
\begin{equation}
\label{eq:fine_propagator}
\left\{
\begin{array}{l}
\left[
\begin{array}{l}
A_1  \Delta{\bf u}^1_k        =  {\bf r}_1 \\
 \vdots  \\
 \prod\limits_{i=1}^{J} A_i \Delta{\bf u}^{J}_k =   \sum\limits_{j=2}^{J}  \prod\limits_{i=1}^{j-1} A_i {\bf r}_j  + {\bf r}_1 \\
\end{array}
\right] \ {\rm Block} \ 1 \\
\left[
\begin{array}{ll}
A_{J+1} \Delta{\bf u}^{J+1}_k        = & {\bf r}_{J+1}  + \left(\Delta{\bf u}^{1}_c\right)_{k-1} \\
 \vdots  \\
 \prod\limits_{i=J+1}^{2J} A_i \Delta{\bf u}^{2J}_k = &  \sum\limits_{j=J+2}^{2J}  \prod\limits_{i=J+1}^{j-1} A_i {\bf r}_j  + {\bf r}_{J+1} \\
\end{array}
\right] \ {\rm Block} \ 2 \\
\centerline{ \vdots}  \\
\left[
\begin{array}{ll}
A_{(M-1)J+1}  \Delta{\bf u}^{(M-1)J+1}_k        =  {\bf r}_{(M-1)J+1} +  \left(\Delta{\bf u}_c^{M-1}\right)_{k-1} \\
 \vdots  \\
 \prod\limits_{i=(M-1)J+1}^{M J} A_i \Delta{\bf u}^{M J}_k =   \sum\limits_{j=(M-1)J+2}^{M J}  \prod\limits_{i=(M-1)J+1}^{j-1} A_i {\bf r}_j  + {\bf r}_{(M-1)J+1} \\
\end{array}
\right] \ {\rm Block} \ M, 
\end{array}
\right.
\end{equation}
where $A_i = I + \tau \left(\frac{\partial{\bf F}}{\partial {\bf u}}\right)^i$ and $k=1,2,\dots$ is the Parareal iteration index.  In contrast to the block Jacobi method,  the $\left(\Delta{\bf u}^{1}_c\right)_{k-1}, \dots, \left(\Delta{\bf u}_c^{M-1}\right)_{k-1}$ terms in Eq.~(\ref{eq:fine_propagator}) are the solution of the coarse-grid correction step (see Section \ref{CGC}) at the previous Parareal iteration.  The key difference between the classical Parareal method and the new method proposed herein is that not only all $M$ blocks are computed in parallel, but also all $J=\frac{N_t}M$ time levels within each block are fully decoupled and calculated in parallel as well, thus providing the highest level of parallelism where every time level is computed on its own computing core simultaneously.

\subsection{Step 3: Coarse Grid Correction}
\label{CGC}
To accelerate the convergence,  we propagate the error between the coarse- and fine-grid solutions by including it into the coarse-grid propagator as a source term.  This coarse grid correction step can be represented as follows: 
\begin{equation}
\label{eq:CGC}
\left\{
\begin{array}{rl}
A^c_1  (\Delta{\bf u}^1_c)_k        = & \tilde{\bf r}^c_1 \\
\prod\limits_{i=1}^{M}A^c_i (\Delta{\bf u}^2_c)_k = &A^c_1 \tilde{\bf r}^c_2 + \tilde{\bf r}^c_1 \\
\vdots & \\
 \prod\limits_{i=1}^{M} \Delta A^c_i (\Delta{\bf u}^{M}_c)_k = &  \sum\limits_{j=2}^{M}  \prod\limits_{i=1}^{j-1} A^c_i \tilde{\bf r}^c_j  + \tilde{\bf r}^c_1, \\
\end{array}
\right.
\end{equation}
where the right-hand side vector $\tilde{\bf r}^c_m$ is given by
$$
\tilde{\bf r}^c_m = {\bf r}^c_m + \Delta{\bf u}^{\frac{N_t}M (m-1)}_{k-1} - (\Delta{\bf u}^{m-1}_c)_{k-1}  \quad {\rm for} \ m = 1,\dots, M 
$$
In contrast to the conventional Parareal method where the coarse grid correction is based on a sequential algorithm that cannot be parallelized,  the above coarse grid correction equations are fully decoupled and can be computed in parallel.

\subsection{Stopping criteria}
Parareal iterations given by Eqs.~(\ref{eq:fine_propagator}-\ref{eq:CGC}) are repeated until convergence. There are several stopping criteria available in the literature.  In the present analysis, the Parareal iterations are stopped if
$$
\| (\Delta{\bf u}_c)_{k-1} - (\Delta{\bf u}_c)_{k-1}  \| < \epsilon_P,
$$
where $\| \cdot \|$ is a suitable norm (e.g., $L_2$ for smooth solutions and $L_1$ for discontinuous solutions) in the entire space-time domain and $\epsilon_P$ is a user-defined tolerance.  Once the above stopping criterion is met,  the solution vector,  the Jacobian matrices, and the right-hand sides are updated as follows:
$$
\begin{array}{ll}
{\bf u}^n_{l+1} &= {\bf u}^n_l  + \Delta{\bf u}^n_l  \\
A_n  &= A_n({\bf u}^n_{l+1}) \\
{\bf r}_n &= {\bf r}_n({\bf u}^n_{l+1}))
\end{array}
\quad {\rm for} \ n=1,\dots,N_t,
$$
where $l$ is the Newton iteration index.  The Newton iterations  are repeated until 
$$
\| \Delta{\bf u}_l  \| < \epsilon_N,
$$
where $\epsilon_N$ is a tolerance for the Newton method.  For all numerical experiments considered, the Parareal and Newton tolerances are computed as follows: $\epsilon_P = C_{\rm sf}\epsilon_N$, where $C_{\rm sf}$ is a safety factor such that $C_{\rm sf} \ll 1$.

\subsection{Speedup Analysis}
\label{sec:speedup}
The parallel efficiency of the ParaDIn-Parareal method presented above strongly  depends on the number of Parareal iterations required for convergence.  Unlike the conventional Parareal method,  the computational cost of the coarse-grid correction and the fine-grid solve are nearly the same for the proposed method, because these steps are implemented using the ParaDIn algorithm for which  each time level is calculated on its own computing core in parallel.  So,  computational costs of the coarse- ($W^c_{\rm par}$) and fine-grid ($W^f_{\rm par}$) solves can be evaluated as follows:
\begin{equation}
\label{work}
\begin{array}{l}
W^c_{\rm par} =  k_N \left[ O(N_s) +  O(J^2 N_s) + W_{\rm sol}  + W_{\rm com}\right], \\
W^f_{\rm par} =  k_N \left[ O(N_s) +  O(M^2 N_s) + W_{\rm sol}  + W_{\rm com}\right],
\end{array}
\end{equation}
where $M$ is a number of blocks (i.e.,  a number of points of the coarse grid) and $J$ is a number of time levels within each block
(i.e.,  a number of points of the fine grid in each block), 
$W_{\rm sol}$ is a computational cost of solving a linear system of equations at each time level,  $N_s$ is a number of spatial degrees of freedom, $W_{\rm com}$ is a communication time between computing cores, and $k_N$ is a Newton iteration index of the parallel-in-time method.  
{In Eq.~(\ref{work}), it is assumed that 
\begin{equation}
\label{cstr}
\begin{array}{ll}
M &\ll O(N_s^{1/2}) \\
J  & \ll O(N_s^{1/2}), 
\end{array}
\end{equation}
so that all the conditions of Theorem 2 in \cite{YP}  hold.} 
Note that the first and second terms in the square brackets in Eq.~(\ref{work}) represent the cost of computing the Jacobian matrix and the left- and right-hand sides of the vector equation at a given time level on each computing core.  As we discuss in Section \ref{GIT},  the overall computational time also includes the cost of computing the initial guess for all time levels, which  is $c_f^p W_{\rm sol}$, where $c_f$ is a space-time coarsening factor, and $p\ge 3$ for 2-D problems.  Combining these computational costs together and assuming that the sequential and parallel algorithms converge in the same number of Newton iterations $k_N$, the speedup $S = {W_{\rm seq}}/{W_{\rm par}}$ that can be obtained with the  ParaDIn--Parareal method is given by
\begin{equation}
\label{eq:Sest}
S = \frac{k_N  N_t  \left[O(N_s) + W_{\rm sol}\right] }{\frac{k_{N} N_{\rm t} W_{\rm sol}}{c_f^{\rm p}}  + k_{N} \left[ O(N_s) +  O((M^2+J^2) N_s) + (2 k_P+1)W_{\rm sol}  + W_{\rm com}\right]},
\end{equation}
where $k_P$ is the number of Parareal iterations.
The factor $(2 k_P +1)$ in the denominator of Eq.~(\ref{eq:Sest}) is due to the fact that one coarse-grid solve and one fine-grid solve are needed per each Parareal iteration.  One more solve is required to initialize the solution on the coarse temporal grid.
Assuming that  $M$ and $J$ satisfy the constraints given by Eq. ~(\ref{cstr}),  we have $W_{\rm sol} \gg O((M^2+J^2) N_s) \gg O(N_s)$ and $W_{\rm sol} \gg W_{\rm com}$.  
Thus,  the speedup provided by the proposed ParaDIn--Parareal method can be estimated as follows:
\begin{equation}
\label{eq:PPspeedup}
S = \frac{N_t}{\frac{N_t}{c_f^p}  + 2 k_P + 1} \approx \frac{N_t}{2 k_P+1}.
\end{equation}
As follows from the above estimate,  the speedup obtained with the ParaDIn--Parareal method strongly depends on the number of Parareal iterations $k_P$ required for convergence.  For example,  if the Parareal method converges in 2 iterations,  i.e., $k_P = 2$ and the coarsening factor  is $c_f >2$, then 
the expected speedup would be $\frac{1}{5}$-th of the total number of computing cores (i.e.,  the total number of time steps $N_t$) used.

\section{ParaDIn--Parareal method with spatial coarsening}
\label{sec:PP_coarse}
As evident from Eq.(\ref{eq:Sest}),  the speedup of the ParaDIn--Parareal method can be increased by about a factor of 2, if we make the computational cost of the initialization and coarse grid correction steps much smaller than that of the fine--grid propagator step.  Note that for the ParaDIn-Parareal method, the temporal coarsening alone is not enough to construct a computationally cheap coarse-grid propagator, because each time level is computed on its own computing core in parallel.   A possible approach to achieve this goal is to use space-time multigrid methods \cite{GN}.  In the present analysis, we use an alternative approach based on the Parareal method with spatial coarsening \cite{Rup, AGR}.  
For the ParaDIn--Parareal method presented in Section \ref{sec:PP},  the coarse- and fine-grid propagators are based on the same spatial grid.  
The main idea of the spatial coarsening approach is to coarsen the spatial grid of the coarse-grid propagator by a factor of $c_s$ , which reduces the computational cost of the initialization and coarse grid correction steps by about a factor of $c_s^d$, where $d$ is the number of spatial  dimensions. 
It should be noted that this spatial coarsening does not affect the overall accuracy,  because the solution of Eq.~(\ref{eq:GlobalNewton}) is obtained during the fine-grid propagator step.

In the ParaDIn--Parareal method with spatial coarsening,  the initialization and the coarse grid correction steps are nearly identical to those given by Eqs. (\ref{eq:ParaDin_init}) and (\ref{eq:CGC}), respectively.  The only difference between the original and spatial coarsening approaches is that the Jacobian matrices and the right-hand sides are defined on the coarse rather than the original spatial grid, i.e.,
\begin{itemize}
 \item[] $A^c_m        = I^c + \Delta t \frac{\partial{\bf F}}{\partial {\bf u}}\left(\mathcal{R}{\bf u}^{Jm}_l\right)$ 
 \item[] $\tilde{\bf r}^c_m = {\bf r}^c\left(\mathcal{R}{\bf u}_l^{Jm}\right) + \mathcal{R}\Delta{\bf u}^{J (m-1)}_{k-1} - (\Delta{\bf u}^{m-1}_c)_{k-1},  \quad {\rm for} \ m = 1,\dots, M$, 
\end{itemize}
where  $\mathcal{R}$ is a restriction operator that transfers the solution from the fine to the coarse grid,  ${\bf u}_l^{Jm}$ is the fine-grid solution from the previous Newton iteration, $I^c$ is an $N^c \times N^c$ identity matrix,   $N^c = \frac{N_s}{c_s}$,  $c_s$ is a spatial coarsening factor, and $J=\frac{N_t}M$.  In the present method, the restriction operator $\mathcal{R}$  is a simple injection that assigns each coarse grid node a value at the corresponding fine grid node.

The fine-grid propagator of the ParaDIn--Parareal method with spatial coarsening is given by
\begin{equation}
\label{eq:fine_coarsening}
\left\{
\begin{array}{l}
\left[
\begin{array}{l}
A_1  \Delta{\bf u}^1_k        =  {\bf r}_1 \\
\vdots  \\
 \prod\limits_{i=1}^{J} A_i \Delta{\bf u}^{J}_k =   \sum\limits_{j=2}^{J}  \prod\limits_{i=1}^{j-1} A_i {\bf r}_j  + {\bf r}_1 \\
\end{array}
\right] \ {\rm Block} \ 1 \\
\left[
\begin{array}{ll}
A_{J+1} \Delta{\bf u}^{J+1}_k        =  {\bf r}_{J+1}  + \mathcal{I}\left(\Delta{\bf u}^{1}_c\right)_{k-1} \\
\vdots  \\
 \prod\limits_{i=J+1}^{2J} A_i \Delta{\bf u}^{2J}_k =   \sum\limits_{j=J+2}^{2J}  \prod\limits_{i=J+1}^{j-1} A_i {\bf r}_j  + {\bf r}_{J+1} \\
\end{array}
\right] \ {\rm Block} \ 2 \\
\centerline{ \vdots}  \\
\left[
\begin{array}{ll}
A_{(M-1)J+1}  \Delta{\bf u}^{(M-1)J+1}_k        =  {\bf r}_{(M-1)J+1} +  \mathcal{I}\left(\Delta{\bf u}_c^{M-1}\right)_{k-1} \\
\vdots & \\
 \prod\limits_{i=(M-1)J+1}^{M J} A_i \Delta{\bf u}^{M J}_k =   \sum\limits_{j=(M-1)J+2}^{M J}  \prod\limits_{i=(M-1)J+1}^{j-1} A_i {\bf r}_j  + {\bf r}_{(M-1)J+1} \\
\end{array}
\right] \ {\rm Block} \ M, 
\end{array}
\right.
\end{equation}
where $\mathcal{I}$ is an interpolation operator that transfers the solution from the coarse to the fine grid,  $A_i = I + \tau \left(\frac{\partial{\bf F}}{\partial {\bf u}}\right)^i$ is an $N_s\times N_s$ Jacobian matrix, and $k$ is a Parareal iteration index.  In the present method,  the coarse-grid solution is interpolated to the fine grid by using a 1-D cubic spline in a direction-by-direction manner.
 
The speedup of the ParaDIn-Parareal method with spatial coarsening can be evaluated in the same way described in Section \ref{sec:speedup}, thus leading to
\begin{equation}
\label{eq:Sest1}
S =
\frac{k_N  N_t  \left[O(N_s) + W_{\rm sol}\right] }{\frac{k_{N} N_{\rm t} }{c_f^{\rm p}}W_{\rm sol}  + k_{N} \left[ O(N_s) +  O((M^2+J^2) N_s) + \frac{(k_P+1)}{c_s^d}W_{\rm sol} + k_P W_{\rm sol}  + W_{\rm com}\right]},
\end{equation}
where $c_s$ is a spatial coarsening factor and $d$ is the number of spatial dimensions. Note that the $(k_P +1)W_{\rm sol}$ term in the denominator is scaled by $\frac1{c_s^d}$,  because the spatial coarsening is used only at the initialization and coarse grid correction steps.  Under the same assumptions outlined in Section \ref{sec:speedup}, the speedup can be evaluated as follows:
\begin{equation}
\label{eq:PPspeedup1}
S = \frac{N_t}{\frac{N_t}{c_f^p}  + \frac{k_P + 1}{c_s^d} +  k_P} \approx \frac{N_t}{k_P}.
\end{equation}
As expected,  the speedup provided by the ParaDIn-Parareal method with spatial coarsening is approximately as twice as large as the speedup of the baseline method.  It should be noted, however, that for convection-dominated problems,  the Parareal method with spatial coarsening may diverge or demonstrate significantly slower convergence rate than that of the baseline method without coarsening \cite{AGR}.  Our numerical results obtained using the ParaDIn--Parareal method with spatial coarsening, which are presented in Section \ref{results},  demonstrate the convergence behavior consistent with the theoretical analysis presented in \cite{AGR}.

\section{Numerical results}
\label{results}

 We now present numerical results obtained with both variants of the ParaDIn-Parareal method, namely,  with and without spatial coarsening for the 2-D nonlinear heat and Burgers equations.  To demonstrate the parallel efficiency of the proposed methods, we also compare them with the Parareal method on the same benchmark problems.  Note that the classical Parareal method diverges for these nonlinear problems.  Therefore, we use the Parareal method to solve the all-at-once system of linear equations (\ref{eq:GlobalNewton}) at each global Newton iteration.

The following two model problems considered in the present study are the 2-D nonlinear heat and Burgers equations. For the nonlinear heat equation, $f = 0$,  $g = 0$  $\forall (x,y) \in \Omega,  t\in (0, T_f]$ and $\mu(u) = \mu_0 u^2$, where $\mu_0$ is a positive constant.  In a 2-D case, the Burgers equations are a system of two coupled equations which are modified such that both velocity components are assumed to be equal each other in the entire space-time domain.  As a result, the system of two equations reduces to Eq. (\ref{eq:CLE}) with the following inviscid fluxes:  $f = g = \frac{u^2}2$.  For the Burgers equation, the viscosity coefficient $\mu=\mu_0$ is assumed to be a positive constant. 
For all numerical experiments presented herein, we set $N_x=N_y$ and use a uniform grid in time,  i.e., $\tau_1 = \tau_2 = \dots = \tau_{N_t }$. 
The system of linear equations at each time level is solved by using a standard direct solver for banded matrices without pivoting.  To make a fair comparison between the corresponding sequential and parallel-in-time schemes, the same direct solver for banded matrices is used for all methods considered.  
The ParaDIn-Parareal scheme has been parallelized only in time, such that the number of computing cores is precisely equal to the total number of time steps $N_t$, i.e., each time step is calculated on one core.
 Also,  all $L_1$,  $L_2$, and $L_{\infty}$ error norms presented in this section are evaluated over the entire space-time domain.  
 For all test problems considered, we run the Parareal, ParaDIn-Parareal, and sequential BDF1 schemes with identical input parameters and  on the same space-time grids and the same computing cores.

\subsection{2-D nonlinear heat equation}
\label{nheat}
To evaluate the parallel performance of the ParaDIn-Parareal method for problems with smooth solutions, we solve the 2-D nonlinear heat equation given by Eq.~(\ref{eq:CLE}) with  $f(u) = g(u)= 0,   \forall (x,y)\in[0.1,1.1]\times[0.1,1.1],  \ t\in[0,1]$ The viscosity coefficient in Eq. (1) is given by $\mu = \mu_0 u^2$, where $\mu_0$ is set equal to $10^{-6}$. This nonlinear heat equation has the following exact solution:
$$
u_{\rm ex}(x,t) = \sqrt{\sqrt{\frac \alpha{\mu_0}}(x + y) + \alpha t  +1},
$$
where $\alpha$ is a positive constant that is set to be $1.0$ for all test cases considered in this section.  The above exact solution is used to define the initial and boundary conditions given by Eq. (\ref{eq:BC}).  An initial guess for the all-at-once Newton method is the solution of this parabolic equation obtained by the corresponding sequential BDF1 scheme on a grid coarsened by a factor of $c_f = 4$ both in each spatial direction and time, which is then  interpolated to the original grid by using the 1-D cubic spline in the direction-by-direction manner.

First, we verify that the proposed ParaDIn-Parareal BDF1 scheme provides the same accuracy as the original sequential counterpart for this test problem.  As evident from Table \ref{conv_heat}, the $L_2$ discretization errors of the ParaDIn-Parareal BDF1 schemes and its sequential counterpart are  identical to each other for all grids considered.  Note that for this test problem, the discretization error is dominated by
its spatial component, so that the BDF1 scheme demonstrates the convergence rate that approaches to 2 as the grid is refined.
\begin{table}[!h]
\begin{center}
\begin{tabular}{ccccc}
\hline
    $N_t\times N_x\times N_y$    &  $M_c$  & Sequential BDF1 &   ParaDIn-Parareal   &  $L_2$ rate\\
                                                      &                & $L_2$ error         &  $L_2$ error             &                     \\
\hline
  $30\times       4\times 4$   &  1  &    $2.40\times 10^{-5}$  &   $2.40\times 10^{-5}$ &  $-$  \\
  $60\times       8\times 8$   &  2  &   $8.50\times 10^{-6}$  &   $8.50\times 10^{-6}$&  $1.50$    \\
  $120\times  16\times 16$   &  4  &    $2.65\times 10^{-6}$  &   $2.65\times 10^{-6}$&  $1.68$  \\
  $240\times 32\times 32$   &   8  &   $7.35\times 10^{-7}$  &   $7.35\times 10^{-7}$&  $1.85$  \\
  $480\times 64\times 64$   &  16 &   $1.88\times 10^{-7}$  &   $1.88\times 10^{-7}$& $1.97$  \\
\hline
\end{tabular}
\end{center}
\caption{\label{conv_heat} $L_2$ errors obtained with the sequential and ParaDIn-Parareal BDF1 schemes for the 2-D nonlinear heat equation on uniformly refined grids.
}
\end{table}
\begin{table}[!h]
\begin{center}
\begin{tabular}{ccccc}
\hline
Number     &  $M_c$  & Sequential BDF1 &   ParaDIn-Parareal  &  Speedup   \\
of cores     &               & runtime (s)           &  runtime  (s)            &                    \\
\hline
  30   &  1  &   3813 .9  &  377.5      &  10.1    \\
  60   &  2  &   8165.8  &  746.2     &  10.9    \\
 120  &  4  &   15119.4  &  707.1      &  21.4    \\
 240 &   8  &  27215.2  &  648.3    & 42.0    \\
 480 &  16 &  81238.2  &  926.9     & 87.6    \\
\hline
\end{tabular}
\end{center}
\caption{\label{speedup_heat} The speedup obtained with the ParaDIn-Parareal BDF1 scheme for the 2-D nonlinear heat equation on 
grids with $N_t = 30,  60,  120, 240, 480$,  and $N_x = N_y = 64$.
}
\label{speedup_heat}
\end{table}

The speedup obtained with the ParaDIn-Parareal BDF1 scheme as compared with the corresponding sequential counterpart for the 2-D nonlinear heat equation is presented in Table \ref{speedup_heat}.  
For this 2-D test problem with the smooth solution,  two Parareal iterations, $k_P$,  are enough to drive the residual below than $10^{-9}$ at each Newton iteration.  Both the ParaDIn-Parareal method and the corresponding sequential scheme converge quadratically in two Newton iterations and reduce the residual below than $10^{-8}$ at each time level.  
As evident from Table \ref{speedup_heat},  the speedup provided by the ParaDIn-Parareal method increases as the number of time steps and consequently the number of processors increases and reaches its maximum value of about $88$ on 480 computing cores,   which is consistent with our estimate given by Eq. (\ref{eq:PPspeedup}).  Note that the maximum speedup that can be obtained with the  baseline Parareal method regardless of the number of processors used is only $4.3$ for the same test problem on the $N_t=480$ grid.

As discussed in Section \ref{sec:PP_coarse},  the parallel efficiency of the ParaDIn-Parareal method can be further improved if the coarse-grid correction and initialization steps are performed on a coarse spatial grid rather than of the original grid. To demonstrate this property, we solve the same 2-D nonlinear heat equation using the ParaDIn-Parareal method whose coarse-grid correction step is performed on a spatial grid that is coarsened by a factor of 2 in each spatial dimension.  Table \ref{speedup_heat_coarse} shows that for this test problem, the spatial coarsening allows us to increase the speedup to 124  on 480 computing cores. 
\begin{table}[!h]
\begin{center}
\begin{tabular}{ccccc}
\hline
Number     &  $M_c$  & Sequential BDF1 &   ParaDIn-Parareal  &  Speedup   \\
of cores     &               & runtime (s)           &  runtime  (s)            &                    \\
\hline
  30   &  1  &   3834 .9  &  286.8    &  13.5    \\
  60   &  2  &   9229.3  &  619.5     &  14.9    \\
 120  &  4  &   15069.7  &  524.5    &  28.7    \\
 240 &   8  &  32679.8  &  559.7    & 58.4    \\
 480 &  16 &  81044.8  &  652.6     & 124.2    \\
\hline
\end{tabular}
\end{center}
\caption{\label{speedup_heat_coarse} The speedup obtained with the ParaDIn-Parareal BDF1 scheme with spatial coarsening for the 2-D nonlinear heat equation on  grids with $N_t = 30,  60,  120, 240, 480$,  and $N_x = N_y = 64$.
}
\label{speedup_heat_coarse}
\end{table}

Though the spatial grid coarsening at the coarse grid correction step does not increase the discretization error of the overall ParaDIn-Parareal BDF1 scheme,  it may slow down the convergence.  Note, however, that this is not the case for this test problem and the ParaDIn-Parareal method with spatial coarsening also converges in two Newton iterations.  It should also be noted that the $L_2$ errors obtained with the sequential scheme and the ParaDIn-Parareal method with coarsening are identical to each other for all grids considered.
Form these results, we can conclude that the ParaDIn-Parareal BDF1 scheme with spatial coarsening allows us to reduce the computational time by up two orders of magnitude as compared with the sequential counterpart without sacrificing the solution accuracy.

\subsection{2-D viscous Burgers equation}
The second test problem is  a viscous shock described by the 2-D viscous Burgers equation given by Eq. ~(\ref{eq:CLE}) with  $f(u) = g(u)=u^2/2.$ For this problem,  the exact solution is given analytically in the following form:
\begin{equation}
\label{eq:Burgers}
u_{\rm ex} = \frac v2 \left[ 1 - \tanh\left( \frac{v(x+y-vt)}{4\mu}\right) \right],
\end{equation}
where $v$ is a shock speed that is set to be $0.5$. The viscosity coefficient $\mu$ is set equal to $10^{-3}$, so that the viscous shock wave given by Eq. (\ref{eq:Burgers}) is not fully resolved on any grid considered.  An initial guess for the ParaDIn-Parareal method is constructed in the same way as in the previous test problem  with the coarsening factor of $c_f=3$ in each spatial direction and time.
\begin{table}[!h]
\begin{center}
\begin{tabular}{ccccc}
\hline
    $N_t\times N_x\times N_y$    &  $M_c$  & Sequential BDF1 &   ParaDIn-Parareal      & $L_1$ rate \\
                                                      &                & $L_1$ error         &  $L_1$ error                 &\\
\hline
  $30\times   7\times 7$        &   3   &    $3.00\times 10^{-2}$  &   $3.00\times 10^{-2}$      &   $-$     \\
  $60\times   11\times  11$    &  6    &   $2.12\times 10^{-2}$  &   $2.12\times 10^{-2}$    & 0.50    \\
  $120\times  21\times 21$   &   10  &    $1.33\times 10^{-2}$  &   $1.33\times 10^{-2}$ & 0.67 \\
  $240\times 41\times 41$   &   16  &   $5.96\times 10^{-3}$  &   $5.96\times 10^{-3}$  & 1.16 \\
  $480\times 81\times 81$   &  24 &   $2.55\times 10^{-3}$  &   $2.55\times 10^{-3}$    & 1.22 \\
\hline
\end{tabular}
\end{center}
\caption{\label{grid_ref_Burgers} $L_1$ errors obtained with the sequential and ParaDIn-Parareal BDF1 schemes for the 2-D Burgers equation on 
uniformly refined grids.
}
\end{table}

\begin{table}[!h]
\begin{center}
\begin{tabular}{ccccc}
\hline
Number     &  $M_c$  & Sequential BDF1 &   ParaDIn-Parareal  &  Speedup   \\
of cores     &               & runtime (s)           &  runtime  (s)            &                    \\
\hline
  60   &  2  &   22858.7  &  1365.7      &  16.7    \\
 120  &  4  &   41079.0   &   1379.6     &  29.8   \\
 240 &   8  &  60678.4   &  1436.5     &  42.2    \\
 480 &  16 &  129843.3 &  1645.4     &  78.9    \\
\hline
\end{tabular}
\end{center}
\caption{\label{speedup_Burgers} The speedup obtained with the ParaDIn-Parareal BDF1 scheme for the 2-D Burgers equation on  grids with $N_t = 60,  120, 240, 480$,  and $N_x = N_y = 81$.
}
\end{table}

As in the previous test case, we compare discretization errors obtained with the sequential BDF1 scheme and its ParaDIn-Parareal counterpart for this 2-D Burgers equation on a sequence of globally refined uniform grids, which is shown in Table \ref{grid_ref_Burgers}.  The $L_1$ discretization errors of the ParaDIn-Parareal method are identical to those computed using the sequential BDF1 method on the corresponding grids and demonstrate the design order of convergence. To evaluate the performance of the proposed ParaDIn-Parareal method for problems with strong shock waves,  we solve the 2-D Burgers equation  by using the sequential and ParaDIn-Parareal schemes on the same set of grids. 
Unlike the previous test problem, both the sequential and parallel-in-time BDF1 schemes do not converge quadratically for this test problem with the discontinuous solution.  On average,  three Newton iterations are needed to drive the residual below $10^{-3}$ which is significantly lower than the corresponding discretization error.
As follows form the numerical results that are presented in Table \ref{speedup_Burgers},  the new ParaDIn-Parareal method provides the speedup of  up to nearly 79 as compared with the sequential BDF1 scheme for this highly nonlinear problem with the strong shock wave.
Similar to the previous test problem,  the highest speedup that can be obtained by the baseline Parareal method for this Burgers equation on the $N_t=480$ grid does not exceed $3.2$ regardless of the number of computing cores used.  Based on these results, we can conclude that the proposed ParaDIn-Parareal method drastically outperforms the Parareal method in terms of parallel efficiency for this nonlinear problem with the  discontinuous solution.

All numerical results presented for the 2-D Burgers equations in this section are obtained without using spatial coarsening.  Unlike the parabolic equations, the Parareal method with spatial coarsening does not converge for this Burgers equation,  
which is consistent with the analysis presented in \cite{AGR} for convection-dominated problems.

\section{Conclusions}
In this paper, we have combined the parallel-in-time direct inverse (ParaDIn) method introduced in \cite{YP}  with the classical Parareal method \cite {LMT} and used it for solving highly nonlinear problems discretized by the implicit first-order backward difference (BDF1) scheme in time and the conventional 2nd-order central finite difference scheme in space. 
To enhance the parallel performance of the ParaDIn method and circumvent the key constraint of this methodology caused by large condition numbers of the product matrices on the left-hand side of the all-at-once system at each Newton iteration,   a block-Jacobi preconditioner is used.  Since the block-Jacobi method alone practically provides no appreciable speedup \cite{DMD},  we accelerate the convergence of Jacobi iterations by using the Parareal method \cite{LMT},  which can be interpreted as a two-level multigrid method in time. In the present study,  the Parareal method is used to solve the all-at-once system of linear equations at each global Newton iteration.  To further improve the performance of the combined ParaDIn-Parareal method,  we implement the Parareal coarse-grid correction step in a fully parallel-in-time manner by using the ParaDIn method.  As a result, both the coarse- and fine-grid correction steps of the Parareal method are performed in parallel, which allows us to solve each time level on its own computing core in parallel regardless of 
the number of degrees of freedom used for discretization of the spatial derivatives. 
Our numerical results show that the proposed ParaDIn-Parareal BDF1 scheme provides the speedup of up to 124 and 79 on 480 computing cores for the 2-D nonlinear heat and  Burgers equations,  respectively.

\bigskip
\noindent{\large\bf Acknowledgments}


\noindent
The second author gratefully acknowledges the support from Department of Defense through the grant W911NF-23-10183.






\end{document}